\documentclass[]{amsart}
\usepackage[utf8]{inputenc}
\usepackage{amsmath}
\usepackage{amssymb}
\usepackage{mathtools}
\usepackage{amsthm}
\usepackage{tikz}
\usepackage{xcolor}
\usepackage{hyperref}
\usepackage{url}
\usepackage{braket}
\usepackage{nccmath}
\usepackage{stmaryrd}

\usepackage{algorithm}
\usepackage{algpseudocode}

\hypersetup{
	colorlinks=true,       
	linkcolor=cyan,          
	citecolor=green,        
	filecolor=magenta,      
	urlcolor=cyan           
}

\numberwithin{equation}{section}
\numberwithin{figure}{section}

\allowdisplaybreaks
\theoremstyle{plain}
\newtheorem{lemma}{Lemma}[section]

\newtheorem{theorem}[lemma]{Theorem}

\theoremstyle{definition}
\newtheorem{definition}[lemma]{Definition}

\newcommand{\Rbb}{\mathbb{R}}

\DeclareSymbolFont{bbold}{U}{bbold}{m}{n}
\DeclareSymbolFontAlphabet{\mathbbold}{bbold}

\renewcommand{\P}{\mathbb{P}}

\newcommand{\F}{\mathcal{F}}

\newcommand{\Wcal}{\mathcal{W}}

\newcommand{\Tfrag}{\tau_{\textnormal{frag}}}
\newcommand{\Tcoal}{\tau_{\textnormal{coal}}}

\DeclareMathOperator{\BM}{BM}

\let\emptyset\varnothing

\begin{document}

    \title[Maximal Germ Couplings of BM with Drift]{Constructing Maximal Germ Couplings~\\of Brownian Motions with Drift}
	
	\author{Sebastian Hummel and Adam Quinn Jaffe}
	
	\maketitle

        \begin{abstract}
            Consider all the possible ways of coupling together two Brownian motions with the same starting position but with different drifts onto the same probability space.
            It is known that there exist couplings which make these processes agree for some random, positive, maximal initial length of time.
            Presently, we provide an explicit, elementary construction of such couplings.
        \end{abstract}


\section{Introduction}
    
    Consider all the possible ways of coupling together two Brownian motions with the same starting position but with different drifts onto the same probability space.
    Let us say that a coupling under which the two processes exactly agree for a (possibly random) positive initial length of time is called a \textit{germ coupling}, and that a coupling under which they exactly agree for the maximal possible initial length of time (in the sense of stochastic domination) is called a \textit{maximal germ coupling}.
    Recent works have shown the existence of couplings with some interesting properties: that germ couplings exist \cite{CouplingDuality}, that maximal germ couplings exist and that the distribution of the length of time on which they agree has a known form \cite{MEXIT, Vollering}, and that there exist simultaneous non-maximal germ couplings for Brownian motions with infinitely many different drifts \cite{BusaniSH,SeppalainenSorensenSH}.
    The goal of the current work is to show that, in fact, one can simultaneously construct maximal germ couplings of Brownian motions with all possible drifts, and this can be done in an explicit and elementary way. 
            
    Before getting into the precise problem statement, our main results, and the motivation from related literature, let us describe the basic element of our constructions:
    Suppose that one is given a Brownian motion $B = (B(t): t\ge 0)$ started from 0 and with drift 0.
    Then, for $\theta\in\Rbb$, draw the line $\ell^{\theta} \coloneqq \{(t,\frac{1}{2}\theta t): t\ge0\}$, and find the last time that $B$ intersects $\ell^{\theta}$.
    Now define a new process $B^{\theta}$ that is equal to $B$ before this time and to the reflection of $B$ across $\ell^{\theta}$ after this time.
    It turns out that this new process is a Brownian motion started from 0 and with drift $\theta$, and that it is maximally germ coupled to $B$.
    See Figure~\ref{fig:basic} for an illustration of this construction.
    
    There are several things to note about the basic construction.
    First, the construction yields a maximal germ coupling and not just a germ coupling.
    Second, the only randomness required is the standard Brownian motion $B$, since we only ever apply deterministic operations to the random path.
    Third, the construction holds simultaneously for all $\theta\in\Rbb$.
    Thus, this simple geometric construction can be applied in clever ways to answer very many questions about maximal germ couplings.

    \begin{figure}
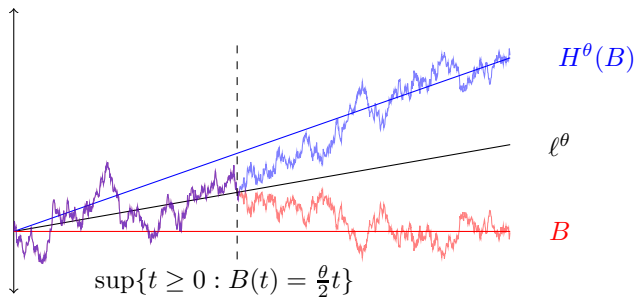

        \begin{center}
            \include{basic_fig}
            \vspace{-0.25cm}
        \end{center}
        \caption{Reflection of a standard Brownian motion $B$ across its last visit to the line $\ell^{\theta}$ yields a Brownian motion $H^{\theta}(B)$ with drift $\theta$.}
        \label{fig:basic}
    \end{figure}

\section{Results}
    
    We begin by establishing our  notation.
    That is, we write $W\coloneqq C([0,\infty);\Rbb)$ for the usual space of continuous real-valued paths endowed with the topology of uniform convergence on compact sets, and we write $\Wcal$ for its Borel $\sigma$-algebra.
    For $\theta\in\Rbb$, write $\BM^{\theta}$ for the probability measure on $(W,\Wcal)$ which is the law of a Brownian motion started at 0 and with drift $\theta$.
    For convenience, we write $\BM \coloneqq \BM^{0}$ for the law of a standard Brownian motion.
    We also write $P$ for a probability measure on an auxiliary probability space supporting a standard Gaussian random variable $Z$.

    Our main results require the notion of fragmentation time.
    That is, for any pair of paths $w_1,w_2\in W$, we define their \textit{fragmentation time} via
    \begin{equation*}
	\Tfrag(w_1,w_2)\coloneqq\inf\{t\ge 0: w_1(t)\neq w_2(t)\}.
    \end{equation*}
    Notice that $\Tfrag:W\times W\to \Rbb$ is a Borel measurable map.

    Next we introduce a few concepts related to coupling.
    That is, for a pair of probability measures $\mu_1,\mu_2$ on $(W,\Wcal)$, a \textit{coupling} of $\mu_1,\mu_2$ is a pair of stochastic processes $X_1,X_2$ defined on a common probability space $(\Omega,\F,\P)$ in such a way that $X_1$ has distribution $\mu_1$ and $X_2$ has distribution $\mu_2$.
    Of course, we are primarily interested couplings of $B,B^{\theta}$ of $\BM,\BM^{\theta}$ which posses further notable structure.

    One structural property at the core of what is to follow is that of a germ coupling:
    A coupling $X_1,X_2$ of $\mu_1,\mu_2$ onto a probability space $(\Omega,\F,\P)$ is called a \textit{germ coupling} of $\mu_1,\mu_2$ if we have $\P(\Tfrag(X_1,X_2)>0) = 1$; alternatively, we say that $X_1,X_2$ are \textit{germ coupled (under $\P$)}.
    In words, a germ coupling is a coupling under which the two stochastic processes are exactly equal for some positive initial length of time.
    
   	Next, we consider germ couplings of Brownian motion that maximize the fragmentation time (in  the sense of stochastic domination). Using the definition of the total variation distance, it can be shown that for any coupling $B,B^{\theta}$ of $\BM,\BM^{\theta}$ onto a probability space $(\Omega,\F,\P)$, we have
	\begin{equation}
	\P(\Tfrag(B,B^{\theta})\le t) \ge P(4\theta^{-2}Z^2\le t)\label{eq:maxgerm}
	\end{equation}
    for all $t\ge 0$; see \cite{MEXIT} for details. 
    Let us say that a coupling satisfying~\eqref{eq:maxgerm} with equality for all $t\ge 0$ is called a \textit{maximal germ coupling}; alternatively, we say that $B,B^{\theta}$ are \textit{maximally germ coupled (under $\P$)}.
    In words, the distribution of $\Tfrag(B,B^{\theta})$ under any coupling is stochastically dominated by the distribution of $4\theta^{-2}Z^2$, and a maximal germ coupling is a coupling under which $\Tfrag(B,B^{\theta})$ and $4\theta^{-2}Z^2$ have the same distribution.

    That the geometric procedure given in the introduction indeed yields a maximal germ coupling is shown in our first result.
    To state it, we introduce for $\theta\in \Rbb$, the mapping of a path $w$ to the path that is $w$ up to the time of the last visit of $w$ to the line $\ell^{\theta}\coloneqq\{(t,\frac{1}{2}\theta t): t\ge 0\}$, and the reflection of~$w$ across $\ell^{\theta}$ after that last visit.
    
    \begin{definition}[Reflection across $\ell^{\theta}$ after last visit]
    For $\theta\in\Rbb$, let $H^{\theta}:W\to W$ denote the map defined via
	\begin{equation*}
		(H^{\theta}(w))(t) \coloneqq \begin{cases}
			w(t), &\textrm{ if } t \le \sup\{s\ge 0: w(s) = \frac{1}{2}\theta s\}, \\
			\theta t - w(t), &\textrm{ else.}
		\end{cases}
	\end{equation*}
    \end{definition}

    As before, see Figure~\ref{fig:basic} for an illustration.
    Now we come to the result:

    \begin{theorem}[Explicit construction of maximal Brownian germ coupling]\label{thm:main}
        If $(\Omega,\F,\P)$ is a probability space on which is defined a process $B$ 
        with distribution $\BM$, then, for all $\theta\in\Rbb$, the process $B^{\theta} := H^{\theta}(B)$ has distribution $\BM^{\theta}$, and $B,B^{\theta}$ are maximally germ coupled under $\P$.
    \end{theorem}

    Let us emphasize that this construction yields a maximal germ coupling of $\BM,\BM^{\theta}$ \emph{simultaneously} for all $\theta\in\Rbb$.
    In particular, it is possible to consider the entire family of processes $\{H^{\theta}(B)\}_{\theta\in\Rbb}$, 
    where $B$ 
    is any $\BM$ on some probability space $(\Omega,\F,\P)$;
    we refer to this ensemble as the \textit{Brownian bouquet} since it contains $B$ as its ``stem'' and $\{H^{\theta}(B)\}_{\theta\neq 0}$ as ``branches'' which, at various random times, fragment from the stem.
    See Figure~\ref{fig:bouquet} for an illustration.

    \begin{figure}
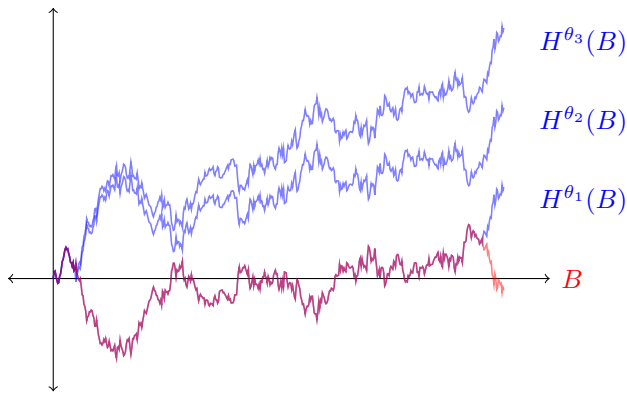

        \begin{center}
            \include{bouquet_fig}
        \end{center}
        \vspace{-0.25cm}
        \caption{The stem and a few branches of the Brownian bouquet.
        That is, we plot a Brownian motion $B$, as well as $H^{\theta}(B)$ for $\theta$ taking on some values $0<\theta_1<\theta_2<\theta_3$.
        The coloring of the stem encodes the ``density'' of how many of the pictured branches have fragmented.}
        \label{fig:bouquet}
    \end{figure}

    A natural next question concerns the stochastic process $\{\Tfrag(B,H^{\theta}(B)): \theta \ge 0)$ which we regard as the ``amount'' of branches not yet fragmented from the stem.
    It turns out that we can describe this process exactly:

    \begin{theorem}[Fragmentation time process]\label{thm:bouquet-branch-rate}
    If $(\Omega,\F,\P)$ is a probability space on which is defined a standard Brownian motion, then the process $(\Tfrag(B,H^{\theta}(B)): \theta \ge 0)$ has the distribution of the reciprocal of a $\frac{1}{2}$-stable subordinator.
    In particular, it is a non-increasing, pure-jump process.
    \end{theorem}

    Our last result addresses the question of how to simulate the pair $B,H^{\theta}(B)$ on a finite interval of time; this is not entirely straightforward, since each point of $H^{\theta}(B)$ depends on the entire trajectory of $B$.
    Nonetheless, we provide in Algorithm~\ref{alg:germ-transform} a deterministic procedure, \textbf{GermTransform}$_{\theta}$, which, as we will soon show, maps suitable random inputs into a maximal germ coupling on a finite interval of time.

    To state this precisely, we need to introduce some modifications of our notation for the finite-time setting.
    So, let us fix $T>0$.
    We write $W(T)\coloneqq C([0,T];\Rbb)$ for the space of continuous real-valued paths endowed with the topology of uniform convergence, and we write $\Wcal(T)$ for its Borel $\sigma$-algebra.
    As before, for $\theta\in\Rbb$ we write $\BM^{\theta}(T)$ for the probability measure on $(W(T),\Wcal(T))$ which is the law of a Brownian motion started at 0 and with drift $\theta$; equivalently, $\BM^{\theta}(T)$ is the restriction of $\BM^{\theta}$ from $(W,\Wcal)$ to $(W(T), \Wcal(T))$.
    We also write $\BM(T) \coloneqq \BM^{0}(T)$ for convenience.

    \begin{theorem}[Simulating maximum Brownian germ coupling]\label{thm:alg}
    Fix $T>0$, and suppose that $(\Omega,\F,\P)$ is a probability space on which is defined a process $B = (B(t): 0 \le t \le T)$ with distribution $\BM(T)$ and a random variable $U$ with distribution $\textnormal{Uniform}[0,1]$, such that $B$ and $U$ are independent.
    Then, the process $B^{\theta}\coloneqq \textnormal{\textbf{GermTransform}}_{\theta}(B,U)$ defined via Algorithm~\ref{alg:germ-transform} has distribution $\BM^{\theta}(T)$, and $B,B^{\theta}$ are maximally germ coupled under $\P$, for all $\theta\ge 0$.
    \end{theorem}

    We regard this result as a sort of ``path-wise Girsanov transform'' which has the additional property of preserving the largest possible initial segment of the Brownian path.
    Indeed, notice the conspicuous appearance of the Radon-Nikod\'ym derivative
    \begin{equation*}
    \frac{d\BM^{\theta}(T)}{d\BM(T)}(w) = \exp\left(\theta w(T)-\frac{1}{2}\theta^2 T\right)
    \end{equation*}
    in line 5 of Algorithm~\ref{alg:germ-transform}.
  
    \begin{algorithm}[t]
	\caption{A deterministic transformation for finite paths. By Theorem~\ref{thm:alg}, if $B = (B(t): 0 \le t \le T)$ is a standard Brownian motion and $U$ is an independent random variable uniformly distributed on $[0,1]$, then \textbf{GermTransform}$_{\theta}(B,U)$ outputs a Brownian motion with drift $\theta$ which is maximally germ coupled to $B$.
    Note that the algorithm requires only one pass through the elements of $B$.}\label{alg:germ-transform}
	
	\begin{algorithmic}[1]
            
		\Procedure{\textbf{GermTransform}$_{\theta}$}{$w,u$}
		  \State \textbf{input:} finite path $w=(w(t): 0 \le t \le T)$ and real number $u\in[0,1]$
		\State \textbf{output:} transformed finite path $(w^{\theta}(t): 0 \le t\le T)$

            \State \,
            
            \If{$u \le \exp(\theta w(T) - \frac{1}{2}\theta^2T)$}
                \State $w^{\theta}\leftarrow w$
            \Else
                \State $t\leftarrow T$
                \While{$w(t) < \frac{1}{2}\theta t$}
                    \State $w^{\theta}(t) \leftarrow \theta t - w(t)$
                    \State decrement $t$
                \EndWhile
                \While{$t\ge 0$}
                    \State $w^{\theta}(t) \leftarrow w(t)$
                    \State decrement $t$
                \EndWhile
            \EndIf
            
            \State \textbf{return} $w^{\theta}$
		\EndProcedure
	\end{algorithmic}
    \end{algorithm}

\section{Related Literature}

    A very classical problem in the theory of stochastic processes is that of \textit{tail coupling}, wherein one asks whether two copies of the same processes started from different starting points can be coupled to be eventually equal almost surely.
    Following the development of a sufficiently general theory \cite{Griffeath, Pitman, Goldstein}, many authors have refined our understanding of the tail coupling problem for certain classes of processes of interest: diffusions \cite{BanerjeeKendall, BenArousCranstonKendall}, integral functionals of Brownian motions and diffusions \cite{KendallArea, KendallPrice}, Brownian motion on Riemannian manifolds \cite{KendallRicci, KuwadaSturm}, Brownian motion in Banach spaces \cite{BMBanach}, L\'evy processes \cite{SchillingLevy2, SchillingLevy1, LevyMarkov}, L\'evy diffusions \cite{LiangSchillingWang}, and more.
    The problem of tail coupling also introduces subtle measurability issues, which have become an interesting question in their own right \cite{HsuSturm, BurdzyKendall}.
    One perspective on the present work is that we provide a refined understanding of a simple version of the dual problem of \textit{germ coupling}, wherein one asks whether two different processes started from the same starting points can be coupled to be initially equal almost surely.
    We believe that a rich avenue for future work is to study the germ coupling problem for other classes of processes of interest, like those outlined above.

    Another related area of research is the study of small-time similarity for stochastic processes.
    For example, a very well-studied notion of small-time similarity is that of ``separating times'' appearing in \cite{SepTimeI, SepTimeII} and another is the result that any stochastic process absolutely continuous with respect to Brownian motion has a local weak limit of Brownian motion at the origin \cite[Lem.~4.3]{DauvergneSarkarVirag}.

    In addition to the thematic connections outlined above, we also have a few explicit precedents for this work as mentioned in the introduction.
    Let us expand on them more fully.

    First is \cite{CouplingDuality} by the second author which studies an abstract question of coupling theory by establishing a form of Kantorovich duality for topologically irregular cost functions.
    As a consequence of the more general theory therein, the following is deduced:
    If $\mu_1,\mu_2$ are any probability measures on the space of c\`adl\`ag paths, then there exists a germ coupling $X_1,X_2$ of $\mu_1,\mu_2$ if and only if we have $\mu_1(A) = \mu_2(A)$ for all $A\in \mathcal{F}_{0+}$, where $\mathcal{F}_{0+}$ denotes the \textit{germ $\sigma$-algebra} on the space of c\`adl\`ag paths \cite[Prop.~2.2]{CouplingDuality}.
    (This correspondence explains the choice of the term \textit{germ coupling}, although there is also a fortuitous connection to the notion of \textit{germ} from algebraic geometry.)
    It is also proven \cite[Thm.~2.4]{CouplingDuality} that the law of a one-dimensional diffusion of the form $dX(t) = \mu(X(t))dt + \sigma(X(t))dB(t)$ admits a germ coupling to a standard Brownian motion if and only if we have $\sigma\equiv 1$ on some neighborhood of $X(0)$.
    While these results are pleasingly general, they are essentially non-constructive.

    Second is \cite{MEXIT} which studies maximal couplings of general stochastic processes.
    (For general processes, there is a notion of maximal coupling which need not be a germ coupling.)
    Their theoretical results are also augmented by a discussion of the statistical implications of the notion (lower bounds for hypothesis testing errors \cite[Prop.~1]{MEXIT}, convergence rates for adaptive Markov chain Monte Carlo \cite[Sect.~2.2]{MEXIT}, and more).
    In addition to many more calculations in discrete-time settings of interest, this work shows that maximal germ couplings of Brownian motions exist, and that $\Tfrag(B,B^{\theta})$, when $B,B^{\theta}$ are maximally germ coupled, is equal in distribution to $4\theta^{-2}Z^2$ where $Z$ is a standard Gaussian random variable \cite[Calculation~1]{MEXIT}.
    The authors provide two constructions of this: one involving a discrete pre-limit from biased random walks \cite[Sect.~5.2]{MEXIT}, and one involving Williams' version of It\^{o} excursion theory for Brownian motions with drift \cite[Thm.~29]{MEXIT}.
    Since these constructions were somewhat involved, we sought, in this work, a more explicit understanding of the structure of maximal germ couplings.

    The third precedent is recent work \cite{BusaniSH,SeppalainenSorensenSH} on the \textit{stationary horizon (SH)} which has been introduced  as a scaling limit of Busemann processes of several models in the Kardar-Parisi-Zhang (KPZ) universality class.
    The stationary horizon is, among other things, an uncountable collection of stochastic processes $\{B^{\theta}\}_{\theta\in \Rbb}$ such that $B^{\theta}$ is a (two-sided) Brownian motion started at 0 and with drift $\theta$ for each $\theta\in\Rbb$, such that $B^{\theta_1}$ and $B^{\theta_2}$ are germ coupled for all $\theta_1,\theta_2\in\Rbb$.
    However, there do not exist any $\theta_1,\theta_2\in\Rbb$ such that $B^{\theta_1}$ and $B^{\theta_2}$ are maximally germ coupled in the SH.
    The construction of the stationary horizon utilizes methods which can in fact be traced back to semi-classical work in queueing theory \cite{OConnellYor}.

    \section{Proofs}

    Before we get into the details, let us give a high-level overview of our method of proof.
    The key observation for all of the proofs is that the \textit{time-inversion} map $I:W\to W$ defined via $(I(w))(t) \coloneqq tw(\frac{1}{t})$ interchanges starting position and drift for Brownian motions; thus, the problem of constructing a \textit{maximal germ coupling} of Brownian motions with starting position zero and with different drifts turns out to be equivalent to the problem of constructing a \textit{minimal tail coupling} of Brownian motions with different starting positions and with zero drift.
    It is known that the so-called \textit{mirror coupling} is a canonical solution to the latter problem, so it only remains to study the image of such couplings under time inversion.
    See Figure~\ref{fig:inversion} for an illustration of this correspondence.

    \begin{figure}
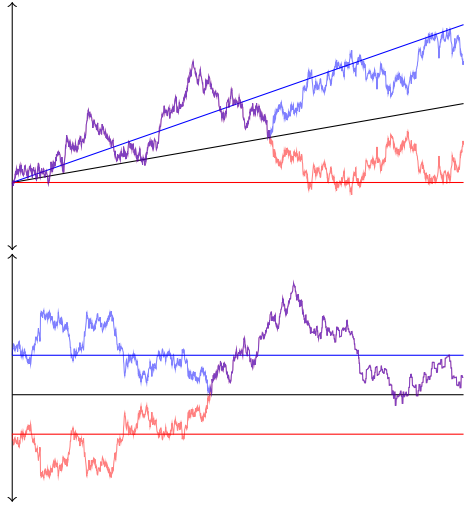

        \begin{center}
            \include{inversion_fig}
        \end{center}
        \vspace{-0.25cm}
        \caption{Using time inversion to restate the problem of constructing a maximal germ coupling as the dual problem of constructing a minimal tail coupling, which is solved by a mirror coupling.}
        \label{fig:inversion}
    \end{figure}

    To make this precise, we define the time-inversion map via
	\begin{equation*}
		I( (w(t): {t\ge 0}) ) \coloneqq (tw(1/t):\, t> 0).
	\end{equation*}
    As a first remark, note that, in general, the value $(I(w))(0)$ may not be well-defined.
    However, if we want the path $I(w)$ to be continuous, then it is necessary that the value $(I(w))(0)$ satisfies $(I(w))(0)=\lim_{u\downarrow 0}(I(w))(u) = \lim_{u\downarrow 0}uw(\frac{1}{u}) = \lim_{u\uparrow\infty}\frac{1}{u}w(u)$.
    Thus, it is natural to consider the restricted path space via
    \begin{equation*}
		W_{\infty}\coloneqq\left\{w\in W: \lim_{u\to\infty}\frac{w(u)}{u} \textrm{ exists}\right\}.
    \end{equation*}
    Then, we can define $I:W_{\infty}\to W_{\infty}$ via
    \begin{equation*}
        (I(w))(t)\coloneqq \begin{cases}
            tw\left(\frac{1}{t}\right), &\textrm{ if } t > 0, \\
            \lim_{u\to\infty}\frac{w(u)}{u},&\textrm{ if } t = 0.
        \end{cases}
    \end{equation*}
    for $w\in W_{\infty}$ and $t\ge 0$.
    It can be easily seen that $I^2(w) = w$ for all $w\in W_{\infty}$, which in other words says that $I$ is an involution on $W_{\infty}$.
    Also observe that the usual strong law of large numbers for Brownian motion implies $\BM(W_{\infty}) = 1$.

    A crucial property of the time-inversion map is that it interchanges the first meeting time with the last meeting time, up to a reciprocal.
	
	\begin{lemma}\label{lem:coal-frag-inversion}
		For any $w_1,w_2\in W_{\infty}$ with $\{s\ge 0: w_1(s) = w_2(s)\}\neq \emptyset$, we have
		\begin{equation*}
			\inf\left\{s\ge 0: (I(w_1))(s) = (I(w_2))(s)\right\} = \frac{1}{\sup\left\{s\ge 0: w_1(s) = w_2(s)\right\}}.
		\end{equation*}
        Here, we use the conventions $1/\infty = 0$ and $1/0 = \infty$.
	\end{lemma}
	
	\begin{proof}
        Note that $\{s\ge 0: w_1(s) = w_2(s)\}\neq \emptyset$ implies $\{s\ge 0: (I(w_1))(s) = (I(w_2))(s)\} \neq \emptyset$, so we do not need to worry about taking the infimum or supremum of an empty set.
        (That the converse is not true is related to the way that we define $(I(w))(0)$ for $w\in W_{\infty}$.
        Indeed, consider $w_1(t)\coloneqq 1$ and $w_2(t) \coloneqq -1$ for $t\ge 0$, so that $\{s\ge 0: w_1(s) = w_2(s)\} = \emptyset$.
        Since $\lim_{u\to\infty}\frac{1}{u}w_1(u) = \lim_{u\to\infty}\frac{1}{u}w_1(u) = 0$, we have $(I(w_1))(0) = (I(w_2))(0)$ hence $\{s\ge 0: (I(w_1))(s) = (I(w_2))(s)\} \neq \emptyset$.)
        If both $\{s\ge 0: w_1(s) = w_2(s)\}$ and $\{s\ge 0: (I(w_1))(s) = (I(w_2))(s)\}$ are non-empty, then the result follows for $0 < \sup\{s\ge 0: w_1(s) = w_2(s)\} < \infty$, simply by observing that for all $u\in\{s>0: w_1(s)=w_2(s)\}$ we have
        \begin{equation*}
			(I(w_1))\left(\frac{1}{u}\right) = \frac{1}{u}w_1(u) = \frac{1}{u}w_2(u) = (I(w_2))\left(\frac{1}{u}\right).
	\end{equation*}
        The extremal cases are handled similarly.
	\end{proof}

        Another crucial property of the time-inversion map is that it interchanges intercept and slope for lines.
        That is, if $w(t) = \theta t + \delta$ for $\theta,\delta\in\mathbb{R}$, then $(I(w))(t) = \delta t + \theta$.
        Interestingly, it has a similar effect on Brownian motions, for which it interchanges starting position and drift.
        To state this precisely, let us define for $\theta,\delta\in \Rbb$ the probability measure $\BM^{\theta}_{\delta}$ on $(W,\Wcal)$ to be the law of a Brownian motion started at $\delta$ and with drift~$\theta$.
        We also omit sub- or super-scripts of $0$ as before, so that, for example, $\BM^{\theta} \coloneqq \BM^{\theta}_0$ denotes the law of a Brownian motion started from $0$ and with drift $\theta\in\Rbb$.
	
	\begin{lemma}\label{lem:inversion-law}
		For any $\theta,\delta\in \Rbb$, we have $\BM^{\theta}_{\delta}\circ I^{-1} = \BM_{\theta}^{\delta}$.
	\end{lemma}
	
	\begin{proof}
		Let $B$ be a standard Brownian motion, so that the process $B_{\delta}^{\theta} \coloneqq( B(t) + \theta t + \delta: t\ge 0)$ has law $\BM^{\theta}_{\delta}$.
		Then, the process $I(B_{\delta}^{\theta})$ is given by
		\begin{equation*}
			(I(B_{\delta}^{\theta}))(t) = t\left(B(1/t)+\theta(1/t)+\delta\right) = tB(1/t) + \delta t + \theta
		\end{equation*}
		for all $t> 0$, and also we have
        \begin{equation*}
            (I(B_{\delta}^{\theta})(0) = \lim_{u\to\infty}\frac{B_{\delta}^{\theta}(u)}{u} = \lim_{u\to\infty}\frac{B(u) + \theta u + \delta}{u} = \lim_{u\to\infty}\frac{B(u)}{u} + \theta =\theta
        \end{equation*}
        almost surely.        
		Since $(tB(1/t): t> 0)$ itself has the law of a standard Brownian motion, we see that $I(B_{\delta}^{\theta})$ has law $\BM_{\theta}^{\delta}$.
	\end{proof}

        With these preliminaries we can now prove our main results.
    
	\begin{proof}[Proof of Theorem~\ref{thm:main}]
		Let $B$ be a Brownian motion.
		By Lemma~\ref{lem:coal-frag-inversion}, we have
		\begin{equation*}
			\inf\left\{s\ge 0: I(B)(s) =  \frac{1}{2}\theta\right\} = \frac{1}{\sup\left\{s\ge 0: B(s) = \frac{1}{2}\theta s\right\}}.
		\end{equation*}
		almost surely.
		Using these two representations, we see that 
		\begin{equation*}
			\begin{split}
				I(H^{\theta}(B))(t) &= t\left(H^{\theta}(B)\left(\frac{1}{t}\right)\right) \\
                    &=\begin{cases}
					tB\left(\frac{1}{t}\right), &\textrm{ if } \frac{1}{t} \le \sup\{s\ge 0: B(s) = \frac{1}{2}\theta s\}, \\
					t\left(\theta \cdot\frac{1}{t} - B(\frac{1}{t})\right), &\textrm{ else.}
				\end{cases} \\
				&= \begin{cases}
					I(B)(t), &\textrm{ if } t\ge \inf\left\{s\ge 0: I(B)(s) =  \frac{1}{2}\theta\right\}, \\
					\theta - I(B)(t), &\textrm{ else.}
				\end{cases}
			\end{split}
		\end{equation*}
		By the usual time-inversion symmetry, the process $I(B)$ has law $\BM$.
		Moreover, by translation symmetry, reflection symmetry, and the strong Markov property, it follows that the process $B_{\theta}\coloneqq I(H^{\theta}(B))$ has the law $\BM_{\theta}$.
		Therefore, Lemma~\ref{lem:inversion-law} implies that $B^{\theta}\coloneqq H^{\theta}(B)$ has law $\BM^{\theta}$. 
		This proves the first part of the theorem.
		
		Next, we claim that $\Tfrag(B^0,B^{\theta}) = \sup\{s\ge 0: B^0(s) = \frac{1}{2}\theta s\}$.
		To see the ``$\ge$'' direction, note by the definition of $H^{\theta}$ that $u\ge 0$ satisfying $u\le\sup\{s\ge 0: B^0(s) = \frac{1}{2}\theta s\}$ implies $B^0(u) = (H^{\theta}(B))(u)$.
		By contrapositive, this means $B^0(u) \neq B^{\theta}(u)$ implies $u>\sup\{s\ge 0: B^0(s) = \frac{1}{2}\theta s\}$.
		Taking the infimum over all such $u$, we conclude $\Tfrag(B^0,B^{\theta}) \ge \sup\{s\ge 0: B^0(s) = \frac{1}{2}\theta s\}$, as desired.
		The ``$\le$'' direction requires more care.
		First note that the definition of $H^{\theta}$ immediately gives
		\begin{equation}\label{eqn:main-1}
			\{\Tfrag(B^0,B^{\theta}) = \infty\} = \left\{\sup\left\{s\ge 0: B^0(s) = \frac{1}{2}\theta s\right\} = \infty\right\}.
		\end{equation}
		Second, observe that, by one-dimensionality, the events
		\begin{align*}
			A^< &\coloneqq \left\{B^0(u)<\frac{1}{2}\theta u \textrm{ for all } u > \sup\left\{s\ge 0: B^0(s) = \frac{1}{2}\theta s\right\}\right\} \\
			A^> &\coloneqq \left\{B^0(u)>\frac{1}{2}\theta u \textrm{ for all } u > \sup\left\{s\ge 0: B^0(s) = \frac{1}{2}\theta s\right\}\right\}
		\end{align*}
		satisfy
		\begin{equation}\label{eqn:main-2}
			A^<\cup A^> = \left\{\sup\left\{s\ge 0: B^0(s) = \frac{1}{2}\theta s\right\} < \infty\right\}.
		\end{equation}
		Thus, we have the following:
		On $A^<$, we have for all $u > \sup\{s\ge 0: B^0(s) = \frac{1}{2}\theta s\}$,
		\begin{equation*}
			B^{\theta}(u) = \theta u- B^0(u) > \theta u - \frac{1}{2}\theta u = \frac{1}{2}\theta u > B^0(u),
		\end{equation*}
		and, on $A^>$, we have for all $u > \sup\{s\ge 0: B^0(s) = \frac{1}{2}\theta s\}$
		\begin{equation*}
			B^{\theta}(u) = \theta u- B^0(u) < \theta u - \frac{1}{2}\theta u = \frac{1}{2}\theta u < B^0(u).
		\end{equation*}
		Consequently, on $A^<\cup A^>$, we can simply take the infimum to get
		\begin{equation*}
			A^{<}\cup A^{>}\subseteq\left\{\Tfrag(B^0,B^{\theta}) \le \sup\left\{s\ge 0: B^0(s) = \frac{1}{2}\theta s\right\} \right\}.
		\end{equation*}
		Combining this with \eqref{eqn:main-1} and \eqref{eqn:main-2}, we conclude $\Tfrag(B^0,B^{\theta}) \le \sup\{s\ge 0: B^0(s) = \frac{1}{2}\theta s\}$, as desired.
		Therefore, we have shown the needed $\Tfrag(B^0,B^{\theta}) = \sup\{s\ge 0: B^0(s) = \frac{1}{2}\theta s\}$.
		
		To complete the proof, we use the conclusion of preceding paragraph, the two given representations of $T$, the fact that $I(B^0)$ is a standard Brownian motion, and lastly the reflection principle, to get
		\begin{align*}
			\P(\Tfrag(B^0,B^{\theta})> t) &= \P\left(\sup\left\{s\ge 0: B^0(s) = \frac{1}{2}\theta s\right\}> t\right) \\
			&= \P\left(\frac{1}{\sup\left\{s\ge 0: B^0(s) = \frac{1}{2}\theta s\right\}}< \frac{1}{t}\right) \\
			&= \P\left(\inf\left\{s\ge 0: B_0(s) = \frac{1}{2}\theta\right\}< \frac{1}{t}\right) \\
			&= \P\left(\sup_{0 \le s \le \frac{1}{t}}B_0(s) > \frac{1}{2}\theta\right) \\
			&= \P\left(\left|B_0\left(\frac{1}{t}\right)\right| > \frac{1}{2}\theta\right) \\
			&= P\left(|Z|>\frac{1}{2}\theta\sqrt{t}\right) \\
			&= P\left(4\theta^{-2}Z^2>t\right)
		\end{align*}
		for all $t\ge 0$.
	\end{proof}

        \begin{proof}[Proof of Theorem~\ref{thm:bouquet-branch-rate}]
        We resume the setting of Theorem~\ref{thm:main}, where we define the processes $B^{\theta} \coloneqq H^{\theta}(B)$ and $B_{\theta} \coloneqq I(B^{\theta})$ for $\theta\ge 0$.
        As before, we use Lemma~\ref{lem:coal-frag-inversion} to get
	\begin{equation*}
		\inf\left\{s\ge 0: B_{0}(s) =  \frac{1}{2}\theta \right\} = \frac{1}{\sup\left\{s\ge 0: B^0(s) = \frac{1}{2}\theta s\right\}}
	\end{equation*}
        for all $\theta \ge 0$, and we also recall that we have shown
        \begin{equation*}
            \Tfrag(B^0,B^{\theta}) = \sup\left\{s\ge 0: B^0(s) = \frac{1}{2}\theta s\right\}
        \end{equation*}
        for all $\theta \ge 0$.
        It is well-known \cite[pg.~116]{RevuzYor} that the process $\{\Tcoal(B_0,B_{\theta}): \theta \ge 0\}$ defined via
        \begin{equation*}
            \Tcoal(B_0,B_{\theta}) \coloneqq \inf\left\{s\ge 0: B^{0}(s) =  \frac{1}{2}\theta \right\}
        \end{equation*}
        for all $\theta \ge 0$ is a $\frac{1}{2}$-stable subordinator, so combining these three displays proves the first claim.
        It immediately follows that $\{\Tfrag(B^0,B^{\theta}): \theta \ge 0\}$ is a non-increasing pure-jump process.
        \end{proof}

 	\begin{proof}[Proof of Theorem~\ref{thm:alg}]
 		The general strategy of the proof is to invoke Theorem~\ref{thm:main} in order to produce a version of $B^\theta$ that is maximally germ coupled to $B$ on $[0,T]$, but the difficulty is that the choice of $B^{\theta} \coloneqq H^{\theta}(B)$ depends on the entire trajectory of $B$ and not only the values $(B(t): 0 \le t \le T)$.
        Even though we do not have access to this entire trajectory, we make the following observations:
        On the event
        \begin{equation}\label{eqn:flip-event}
            \left\{\sup \left\{t\geq 0: \, B(t) = \frac{1}{2}\theta t\right\}> T \right\}
        \end{equation}
        we have $(B^{\theta}(t): 0 \le t \le T) = (B(t): 0 \le t \le T)$ and on its complement we have $(B^{\theta}(t): 0 \le t \le T) = H^{\theta}(B(t): 0 \le t \le T)$, where we have defined
        \begin{equation*}
		(H^{\theta}(w))(t) \coloneqq \begin{cases}
			w(t), &\textrm{ if } t \le \sup\{0 \le s \le T: w(s) = \frac{1}{2}\theta s\}, \\
			\theta t - w(t), &\textrm{ else.}
		\end{cases}
	\end{equation*}
        for $w\in W(T)$ and $0 \le t \le T$.
        It is clear that, in Algorithm~\ref{alg:germ-transform}, the ``if'' clause beginning in line 5 implements the former case, and the ``else'' clause beginning in line 7 implements the latter case; moreover, observe that each case requires only one pass through the elements of $(B(t): 0 \le t \le T)$.
        Therefore, it suffices to show that the conditional probability of \eqref{eqn:flip-event} given $(B(t): 0 \le t \le T)$ is equal to $\min\{\exp(\theta B(T) - \frac{1}{2}\theta^2T),1\}$.

        To do this, we define the process $\tilde{B}$ via $\tilde{B}(s)\coloneqq B(s+T)-B(T)$ for $s\ge 0$ and we do some simple rearranging to get:
        \begin{align*}
 		&\left\{\sup \left\{t\geq 0: \, B(t) = \frac{1}{2}\theta t\right\}> T \right\} \\
            &\qquad=\left\{\inf \left\{s>0: \, B(s+T) = \frac{1}{2}\theta (s+T)\right\} < \infty \right\} \\
            &\qquad=\left\{\inf \left\{s>0: \, \tilde{B}(s)-\frac{1}{2}\theta s = \frac{1}{2}\theta T-B(T)\right\} < \infty \right\} \\
            &\qquad=\left\{\sup_{s>0}\left(\tilde{B}(s)-\frac{1}{2}\theta s\right) \geq \frac{1}{2}\theta T - B(T) \right\}
 	\end{align*}
        almost surely.
        By the strong Markov property, the process $\{\tilde{B}(s)-\frac{1}{2}\theta s\}_{s\ge 0}$ has distribution $\BM^{-\theta/2}$ and is independent of $(B(t): 0 \le t \le T)$.
        Moreover, it is well-known \cite[Cor.~2(i)]{PitmanRogers} that, for $\theta > 0$, the global maximum $\sup_{s>0}(\tilde{B}(s)-\frac{1}{2}\theta s)$ has an exponential distribution with parameter $\theta$.
        Therefore, we have
        \begin{align*}
            &\P\left(\sup \left\{t\geq 0: \, B(t) = \frac{1}{2}\theta t\right\}> T \,\Big|\, B(t) : 0 \le t \le T\right) \\
            &=\P\left(\sup_{s>0}\left(\tilde{B}(s)-\frac{1}{2}\theta s\right) \geq \frac{1}{2}\theta T - B(T) \,\Big|\, B(t) : 0 \le t \le T\right) \\
            &= \begin{cases}
                \exp\left(-\theta \left(\frac{1}{2}\theta T - B(T)\right)\right), &\textrm{ if } B(T) < \frac{1}{2}\theta T, \\
                1, &\textrm{ if } B(T) > \frac{1}{2}\theta T, \\
            \end{cases} \\
            &= \min\left\{\exp\left(\theta B(T) - \frac{1}{2}\theta^2T\right),1\right\}.
        \end{align*}
        as claimed.
 	\end{proof}

        \section*{Acknowledgments}
	       SH was funded by the Deutsche Forschungsgemeinschaft (DFG, German Research Foundation) -- Projektnummer 449823447.
            We also thank Yang Chu, Mehdi Ouaki, Jim Pitman, and Evan Sorensen for useful conversations.

	\nocite{*}
	\bibliography{BM_germ}
	\bibliographystyle{plain}
	
\end{document}